\documentclass[12pt]{amsart}

\usepackage{amsthm, wrapfig, graphicx}

\newtheorem{theorem}{Theorem}[section]
\newtheorem{lemma}{Lemma}[section]

\theoremstyle{definition}
\newtheorem{definition}{Definition}[section]

\newcommand{\nice}{\displaystyle}

\newcommand{\ZZ}{\mathbb{Z}}

\newcommand{\CE}[1]{CE_k(#1)}
\newcommand{\CV}[1]{CV_k(#1)}
\newcommand{\CM}[1]{CM_k(#1)}
\newcommand{\di}[1]{diam(#1)}
\newcommand{\Trl}{T_{r,l}}
\newcommand{\vij}{v_{i,j}}
\newcommand{\vijp}{v_{i,j'}}
\newcommand{\vipjp}{v_{i',j'}}

\newcommand{\floor}[1]{\left\lfloor #1 \right\rfloor}

\newcommand{\ceil}[1]{\left\lceil #1 \right\rceil}

\def\qed{\hfill\hbox{${\vcenter{\vbox{
 \hrule height 0.4pt\hbox{\vrule width 0.4pt height 6pt
 \kern5pt\vrule width 0.4pt}\hrule height 0.4pt}}}$}}
 
\begin{document}

\title{Vertex and Mixed $k$-Diameter Component Connectivity}

\author{Adam Buzzard*}
\address{Moravian College (Undergraduate)}
\curraddr{1200 Main Street Bethlehem, PA 18018}
\email{stawb01@moravian.edu}

\author{Nathan Shank}
\address{Moravian College}
\curraddr{1200 Main Street Bethlehem, PA 18018}
\email{shank@moravian.edu}

\date{(May 2021), and in revised form (July 2021)}
\subjclass[2010]{Primary. 94C15}
\keywords{network reliability, connectivity, conditional connectivity, edge failure, vertex connectivity, mixed connectivity, graph theory}

\thanks{}

\begin{abstract}
In the $k$-diameter component connectivity model a network is consider operational if there is a component with diameter at least $k$.  Therefore, a network is in a failure state if every component has diameter less than $k$.  In this paper we find the vertex variant of the $k$-diameter component connectivity parameter, which is the minimum number of vertex deletions in order to put a network into a failure state, for particular classes of graphs. We also show the mixed variant by allowing vertex and edge failures within the network.  We show results for paths, cycles, complete, and complete bipartite graphs for both variants as well as perfect $r$-ary trees for the vertex variant.  
\end{abstract}

\maketitle

%**************************************************************
%**************************************************************

\section{Introduction}
Many different network structures can be modeled through graph theory. We think of nodes, hubs, people, stations, objects, etc. as vertices of a graph and the communication or connection between them as the edges of the graph. For many reasons, accidental and deliberate, these networks break or fail. Therefore, understanding the reliability and vulnerability of a network is crucial to maintaining and building a reliable network . 

When considering network reliability there are two issues to understand: what are the minimum requirements to maintain an operational network and what pieces of the network may fail. In a network modeled as a graph we can have edges, vertices, or both (mixed) fail. To keep a network operational we often consider what characterizes a failure state for a network. Several different network reliability models have been studied and Harary \cite{Harary1983} provided the general framework for these network reliability models by considering a property $P$ and a network, $G$. He defined a network to be operational if there is a component of $G$ which has property $P$ and therefore, if no component of $G$ has property $P$ the network is in a failure state. Therefore, the network reliability of $G$ based on property $P$ is the minimum number of failures so that no component of $G$ has property $P$. 

For example, the \textit{component order edge connectivity} \cite{Gross2006} considers the minimum number of edges that need to be removed from a graph so that all the components have order less that some specific bound. Similarly the \textit{component order vertex connectivity} \cite{Gross1998} considers vertex removal. For a survey of results see \cite{NET20300}. 

In this paper we will consider a graph operational if it contains a component of diameter at least $k$ for some fixed positive integer $k$. In practice, a network may need to have a component with a minimum diameter for reliability testing of a particular function, the spread of information must travel a minimum distance before it is deemed valuable, or a virus which is transmitted to neighbors may stay dormant until it has passed through a specific number of hosts or nodes. These, and other examples, motivate the need to consider networks which contain a component of minimum diameter. In \cite{ShankBuzz}, the authors considered the instance where edges fail, or are removed. In this paper we consider vertex failure as well as the mixed failure case for certain graph classes. 

\section{Background and Definitions}

We will be using common graph theory notation found in \cite{West2001}. Throughout we will assume that $G = (V,E)$ is a finite simple graph with vertex set $V$ and edge set $E$. For any edge set $D \subseteq E$, let $G-D$ denote the spanning subgraph of $G$ containing the vertex set $V$ and the edge set $E-D$. For any vertex set $H \subseteq V$, let $G-H$ denote the subgraph of $G$ induced by $V-H$. Similarly, if $V' \subseteq V$ and $E' \subseteq E(G-V')$ we will write $G-V'-E'$ to denote $(G-V')-E'$. For any set $A$, let $|A|$ denote the cardinality of $A$. 

If $u, v \in V$, let $d_G(u,v)$ denote the distance between $u$ and $v$ in $G$ (length of the shortest $u-v$ path in $G$). If the graph $G$ is clear, we will denote $d_G(u,v) = d(u,v)$. If $d(u,v) = k$ for some positive integer $k$, then we will say that $u$ is a \textit{k-neighbor} of $v$, $\{u,v\}$ a $k-pair$, and a $u-v$ path of length $k$ will be a \textit{$k$-path}. If there exists a vertex $x \in V$ so that $d(u,x) = k$, then we will say that \textit{$u$ has a $k-neighbor$} in $G$. If $G$ is a connected graph then the \textit{diameter of $G$} is the maximum distance between any two vertices.  If $G$ is not connected, the diameter is defined to be infinite.  A \text{component} of a graph $G$ is a connected and induced subgraph of $G$, call it $H$, so that no other vertex in $v \in V(G-H)$ is adjacent to a vertex in $H$.  Clearly, if $G$ has more than one component, then the diameter of $G$ is infinite.  

Throughout we will consider $k$ to be a positive integer and we will consider a network to be \textit{operational} if there is a component of diameter at least $k$. Thus, a network is in a \textit{failure state} if every component has diameter less than $k$. We can easily render a network into a failure state by removing all of the vertices or all of the edges. However, we are interested in finding the minimum number of vertex or edge deletions to produce a failure state. 

Clearly, if $k\geq 2$ and if a graph contains a $k$-pair, then the graph also contains a $(k-1)$-pair. This leads to the following lemma which will be useful for our considerations of vertex and edge connectivity.  

\begin{lemma}
\label{ifandonly}
Let $G=(V,E)$ be a graph and $k$ be a positive integer. The graph $G$ is in a \textit{failure state} if and only if there does not exist a $k$-pair in $G$. 
\end{lemma}

Thus, in the process of making a failure state we must remove vertices or edges that impact each $k$-path. The following lemma shows that vertex disjoint $k$-paths can not both be impacted by only one vertex or edge removal. This will help us to show bounds on the number of vertex or edge removals we need to render our network into a failure state. 

\begin{lemma}
\label{firsttheorem}
Let $G=(V,E)$ be a graph and $k$ be a positive integer. If there exists $m$ vertex disjoint $k$-paths in $G$, then for any $v \in V$ or $e \in E$, $G-\{v\}$ and $G-\{e\}$ each have at least $m-1$ vertex disjoint $k$-paths.
\end{lemma}

\begin{proof}
Let $G=(V,E)$ be a graph and $k$ be a positive integer. Assume $G$ has $m$ vertex disjoint $k$-paths. Assume by way of contradiction that there exists some $v \in V$ such that $G-\{v\}$ has less than $m-1$ vertex disjoint $k$-paths. Then $v$ was a vertex in at least two of the vertex disjoint $k$-paths. Hence, $G$ did not contain $m$ vertex disjoint $k$-paths. A similar argument holds for edge removals. 

\end{proof}

It is often the case that edges are the object that fails. The \textit{$k$-diameter component edge connectivity} of the graph was introduced in \cite{ShankBuzz} and results were shown for path, complete, and complete bipartite graphs as well as perfect $r$-ary trees. In this paper we focus on the \textit{$k$-diameter component vertex connectivity} parameter as well as the mixed parameter which allows vertex and edge deletions.

\begin{definition}
Let $G=(V,E)$ be a graph and $k$ be a positive integer. A set $V' \subseteq V$ is a \textit{$k$-diameter component vertex disconnecting set} if $G-V'$ has no vertex with a $k$-neighbor.
\end{definition}

This means that a vertex set $V'$ is a $k$-diameter component vertex disconnecting set if every component of $G-V'$ has diameter less than $k$. If $V'$ is a $k$-diameter component vertex disconnecting set, then $G-V'$ is in a failure state.

Recall we are interested in finding the minimum number of vertices that can be removed to produce a failure state. This motivates the definition of the \textit{$k$-diameter component vertex connectivity parameter}.

\begin{definition}
Given a graph $G=(V,E)$ and a positive integer $k$, the \textit{$k$-diameter component vertex connectivity parameter} of $G$, denoted $\CV{G}$, is the size of the smallest $k$-diameter component vertex disconnecting set. 
\end{definition}

Thus, the $k$-diameter component vertex connectivity parameter is the size of the smallest vertex set $V'$ so that $G-V'$ is in a failure state. 

Similarly, we will consider edge disconnecting sets which will be used in our mixed deletion case. 

\begin{definition}
Let $G=(V,E)$ be a graph and $k$ be a positive integer. A set $E' \subseteq E$ is a \textit{$k$-diameter component edge disconnecting set} if $G-E'$ has no vertex with a $k$-neighbor. 
\end{definition} 

This means that an edge set $E'$ is a $k$-diameter component edge disconnecting set if every component of $G-E'$ has diameter less than $k$. If $E'$ is a$k$-diameter component edge disconnecting set, then $G-E'$ is in a failure state. 
 
 As with vertex deletions, we can also consider the minim number of edges whose removal produces a failure state. This motivates the definition of the \textit{$k$-diameter component edge connectivity parameter}. 
 
\begin{definition}
Given a graph $G=(V,E)$ and a positive integer $k$, the \textit{$k$-diameter component edge connectivity parameter} of $G$, denoted $\CE{G}$, is the size of the smallest $k$-diameter component edge disconnecting set.
\end{definition} 

Thus, the $k$-diameter component edge connectivity parameter is the minimum size of an edge set $E'$ so that $G-E'$ is in a failure state. 

It is often the case that vertices and edges fail which is investigated through the mixed deletion case. This was first introduced by Beineke and Harary \cite{connfunc}. The following definitions address the $k$-diameter component connectivity function, which is a mixed version of the $k$-diameter component connectivity involving both vertex and edge deletions. As is standard we will remove vertices first then remove edges. 

\begin{definition}
\label{minedge}
Let $G=(V,E)$ be a graph, $k$ be a positive integer, and $p \in \{0,1,...,\CV{G}\}$. Then the \textit{$k$-diameter component connectivity function} of $G$ is defined as $\CM{G,p}=\text{min}\{\CE{G-V'} : V' \subseteq V, |V'|=p\}$. 
\end{definition}

So $\CM{G,p}$ is the minimum number of edges that must be removed to render the graph into a failure state assuming we can also remove any $p$ vertices in the graph. Note that we must remove the $p$ vertices first then remove the least amount of edges. 

\begin{definition}
Let $G=(V,E)$ be a graph and $k$ be a positive integer. A \textit{$k$-diameter component connectivity pair} of $G$ for each $p \in \{0,1,...,\CV{G}\}$ is an ordered pair $(p,q)$, such that $\CM{G,p}=q$. 
\end{definition}

Two obvious connectivity pairs of $G$ are $(0,\CE{G})$ and $(\CV{G},0)$. For each value of $p$ where $0\leq p \leq \CV{G}$, there is a unique $k$-diameter component connectivity pair.

When trying to find the value of $\CM{G,p}$ it is often useful to consider which $p$ vertices we need to remove to minimize $\CE{G-V'}$. This motivates the following definition of a \textit{optimal $p$-set}. 

\begin{definition}
Let $G=(V,E)$ be a graph, $k$ be a positive integer, and $p$ be a nonnegative integer. Let $V' \subseteq V$ such that $|V'|=p$. We say $V'$ is an \textit{optimal $p$-set} if $\CM{G,p}=\CE{G-V'}$. 
\end{definition}

The following lemma will prove valuable for providing lower bounds. The lemma shows disjoint $k$-paths provide a lower bound for the number of vertex deletions, edge deletions, or mixed deletions needed to produce a failure state.

\begin{lemma} \label{lowerbound}
Let $G=(V,E)$ be a graph and let $k$ be a positive integer. If there exists $M$ vertex disjoint $k$-paths in $G$, then $\CV{G} \geq M$ and $\CE{G} \geq M$. Furthermore, if $CM_k(G,p)=q$, then $p+q \geq M$. 
\end{lemma}

\begin{proof}
Let $G=(V,E)$ be a graph and let $k$ be a positive integer. Assume there exists $M$ vertex disjoint $k$-paths in $G$. Let $V' \subseteq V$ such that $G-V'$ is in a failure state. Let $E' \subseteq E$ such that $G-E'$ is in a failure state. Then Lemma \ref{ifandonly} and multiple iterations of Lemma \ref{firsttheorem} implies that $|V'| \geq M$ and $|E'| \geq M$. 

Let $V^* \subseteq V$ and $E^* \subseteq E(G-V^*)$ such that $G-V^*-E^*$ is in a failure state. Then Lemma \ref{ifandonly} and Lemma \ref{firsttheorem} implies that $|V^*|+|E^*| \geq M$. Hence, if $CM_k(G,p)=q$, then $p+q \geq M$. 

\end{proof}

%**************************************************************
%**************************************************************
%**************************************************************

%Vertex Deletion Starts Here

%**************************************************************
%**************************************************************
%**************************************************************

\section{Vertex Deletion Results}
In this section we will consider only vertex deletions; we compute $\CV{G}$ for specific graphs $G$. We provide results for path graphs, cycles, complete graphs, complete bipartite graphs, and perfect $r$-ary trees. Note that if $k=1$, then $CV_k$ is the minimum number of vertex deletions whose removal results in an edgeless graph. Therefore, we will always assume $k \geq 2$.

\subsection{Path Graphs}
	Consider the path graph on $n$ vertices, denoted $P_n$. Label the vertices consecutively from 1 to $n$ starting at a pendant vertex. Since any path of length $k$ has $k+1$ vertices, there are $\floor{\frac{n}{k+1}}$ vertex disjoint $k$-paths in $P_n$. By Lemma \ref{lowerbound}, $\CV{P_n} \geq \floor{\frac{n}{k+1}}$.
	
If we delete every vertex whose label is a multiple of $k+1$, then all of the remaining components have $k$ vertices, with the exception of at most one component which could have fewer than $k$ vertices. Therefore, the diameter of each remaining component will be less than $k$. This results in a total of $\floor{\frac{n}{k+1}}$ deletions.
Hence, $\CV{P_n} \leq \floor{\frac{n}{k+1}}$. These two observations imply the following:

\begin{theorem}\label{vertexpath}
 For every positive integer n, $$\CV{P_n}= \floor{\frac{n}{k+1}}.$$ 
\end{theorem}

\subsection{Cycle Graphs}
	Consider the cycle graph on $n$ vertices, denoted $C_n$. Since $\di{C_n}=\floor{\frac{n}{2}}$, if $k > \floor{\frac{n}{2}}$, then $C_n$ is already in a failure state and no deletions are necessary. If $k \leq \floor{\frac{n}{2}}$, then at least one deletion must be made. Notice that the deletion of any single vertex from $C_n$ leaves a path graph on $n-1$ vertices. Then, by Theorem \ref{vertexpath}, $\CV{C_n}=\floor{\frac{n-1}{k+1}}+1$. Hence, we have the following:

\begin{theorem}
\label{vertexcycle}
For every positive integer n, $$\CV{C_n}= \begin{cases} 
					 0 & \text{if } k > \floor{\frac{n}{2}} \\
					 \floor{\frac{n+k}{k+1}}& \text{if } k \leq \floor{\frac{n}{2}}. \\
 \end{cases}$$
\end{theorem}

\subsection{Complete Graphs}
	Consider the complete graph on $n$ vertices, denoted $K_n$. Since the diameter of a complete graph is 1 and $k \geq 2$, $K_n$ is already in a failure state. Thus, we see the following obvious result:
	
	\begin{theorem}
	\label{vertexcomplete} 
	For any positive integer $n$, $$\CV{K_n}=0.$$
 \end{theorem}

\subsection{Complete Bipartite Graph}
Now we will consider a complete bipartite graph $K_{a,b}=(V,E)$ with parts $A$ and $B$ where $V=A \cup B$, $A \cap B =\emptyset$, $|A|=a>0$, and $|B|=b>0$. 

\begin{theorem}
\label{vertexcompbipart}
For any positive integer $a$ and $b$,

$$\CV{K_{a,b}}=\begin{cases} 
					 0 & \text{if } a=b=1, \\
					 0 & \text{if } k>2  \text{ and }\max \{a, b\} \geq 2, \\
					 \text{min}\{a,b\} & \text{if } k=2 \text{ and } \max\{a, b\} \geq 2.  \\
 \end{cases}$$
\end{theorem}
\begin{proof}
Let $K_{a,b}=(V,E)$ be a complete bipartite graph with parts $A$ and $B$ where $V=A \cup B$, $A \cap B =\emptyset$, $|A|=a>0$, and $|B|=b>0$. The diameter of a complete bipartite graph is $2$ unless $a=b=1$, in which case the diameter is 1 and $K_{1, 1}$ is already in a failure state for all $k>1$. Consider when $a \geq 2$ or $b \geq 2$. If $k>2$, then $K_{a,b}$ is already in a failure state. Now consider when $k=2$. The only induced subgraphs of $K_{a,b}$ which are in a failure state are $K_{1,1}$, subgraphs of $A$, and subgraphs of $B$. To produce $K_{1,1}$, we must delete all but two vertices: one vertex from $A$ and one vertex from $B$. Thus, the resulting number of vertex deletions is $(a-1) + (b-1)$. Since $A$ is the subgraph of $A$ with the most vertices, we only need to consider deleting vertices to produce $A$. In order to produce $A$, we must delete all vertices from $B$ and, therefore, the resulting number of vertex deletions is $b$. Similarly, to produce $B$, $a$ vertex deletions are necessary. If either $a \geq 2$ or $b \geq 2$, then it is easily seen that $(a-1)+(b-1) \geq \text{min}\{a,b\}$. 

\end{proof}

%%%%%%VERTEX TREE

\subsection{Perfect $r$-ary Trees}

Throughout we will assume $r$ and $l$ are positive integers and let $\Trl = (V,E)$ be a perfect $r-ary$ tree with height $l$. This means that $\Trl$ has $\frac{r^{l+1}-1}{r-1}$ vertices and $\frac{r(r^{l+1}-1)}{r-1}-r$ edges. We can enumerate the vertices and edges of $\Trl$ as follows: 
\begin{align*}
	V &= \{v_{i,j}: 1\leq i \leq l+1 , 1 \leq j \leq r^{(l+1)-i}\}, \text{ and } \\
 E&=\{(v_{i,j},v_{i-1,m}): 2 \leq i \leq l+1, 1 \leq j \leq r^{(l+1)-I},\\
&\indent \indent  (j-1)r+1 \leq m \leq jr \}.
\end{align*}

We will say that vertex $v_{i,j} \in V$ is on \textit{level} $i$. Notice that the root vertex is on level $l+1$ and the leaves are on level $1$. 

\begin{figure}[h]\label{complete2ary}
\includegraphics[width=.7\textwidth]{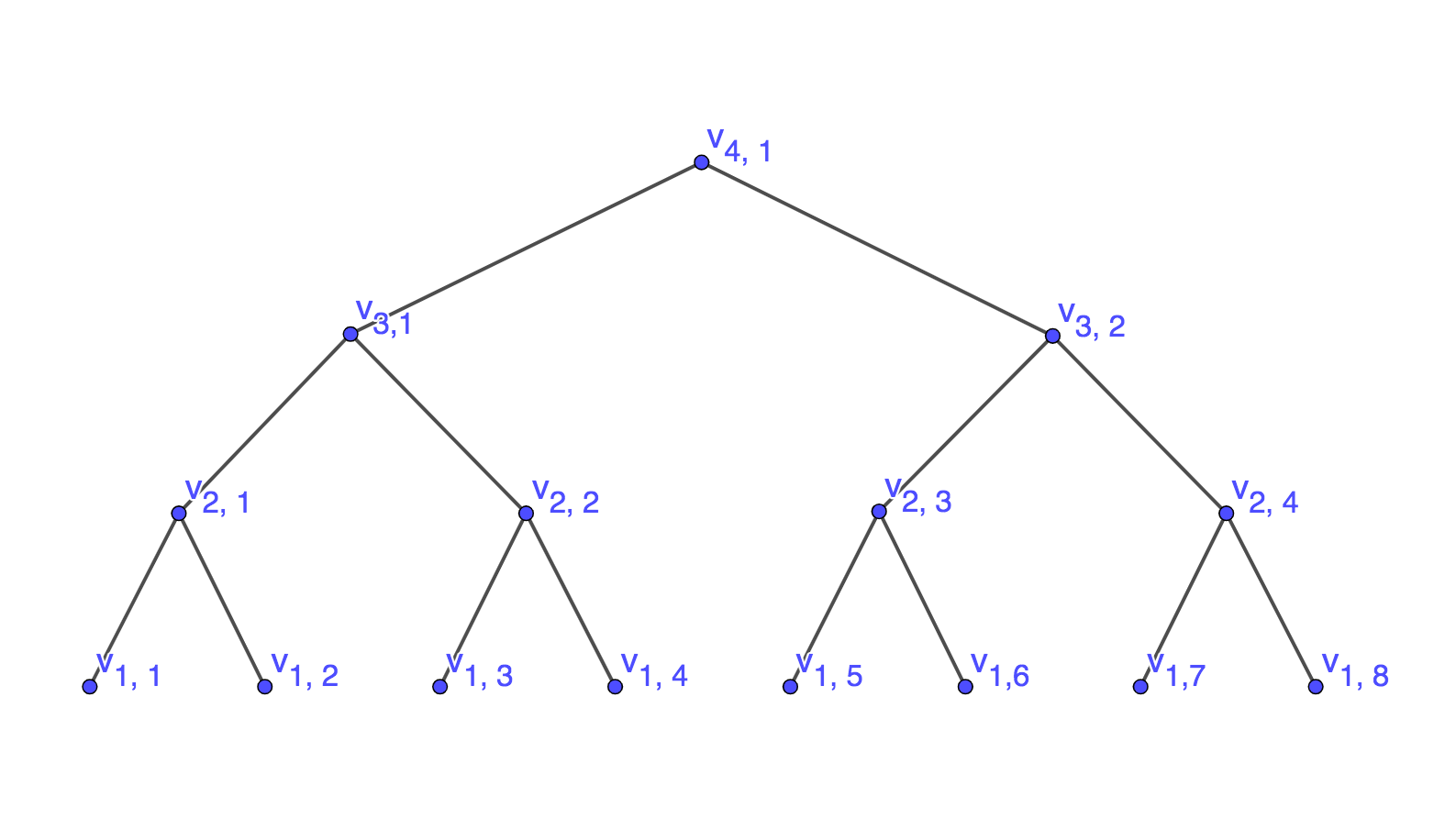}
\caption{A complete 2-ary tree with height 3 ($r=2, l=3$)}
\end{figure}

Let $h$ be a positive integer. Fix a vertex $v_{i,j} \in V$, with $i>h$ and consider the subtree $T_{v_{i,j}}^h$ of $\Trl$ induced by $v_{i,j}$ and all of its descendants at a distance at most $h$. Notice in $T_{v_{i,j}}^h$, the degree of $v_{i,j}$ is $r$ and any vertex $x$ of degree 1 in $T_{v_{i,j}}^h$ will satisfy $d(x,v_{i,j}) = h$. Also notice that $T_{v_{i,j}}^h$ is a perfect $r$-ary tree of height $h$.

The following lemma establishes a set $V'$ of vertices which will form our minimum $k$-diameter component vertex disconnecting set for $\Trl$. The cardinality of this set is shown so that we can prove it is in fact the minimum $k$-diameter component vertex disconnecting set. 

\begin{lemma}
\label{cardvprimevertex}
Let $\Trl=(V,E)$. Let $ V' \subset V$ such that 
$$ V'=\left\{v_{m \left(\ceil{\frac{k}{2}}+1\right),j} \in V: 1 \leq m \leq \floor{\frac{l+1}{\ceil{\frac{k}{2}}+1}}, 1\leq j \leq r^{l+1-m\left(\ceil{\frac{k}{2}}+1\right)}\right\}.$$ Then, $$|V'|=\frac{r^{l+1}-r^{l+1-\floor{\frac{l+1}{\ceil{\frac{k}{2}
}+1}}\left(\ceil{\frac{k}{2}}+1\right)}}{r^{\ceil{\frac{k}{2}}+1}-1}.$$ 
\end{lemma}

\begin{proof} 
Consider $V'$ as defined above. Then by summing over all possible choices of $m$ we see $$\nice |V'|=\sum_{m=1}^{\floor{\frac{l+1}{\ceil{\frac{k}{2}}+1}}} r^{l+1-m\left(\ceil{\frac{k}{2}}+1\right)}.$$ 
Simplifying the previous expression, we see:

\begin{align*}
|V'| &=r^{l+1}\sum_{m=1}^{\floor{\frac{l+1}{\ceil{\frac{k}{2}}+1}}} r^{-m\left(\ceil{\frac{k}{2}}+1\right)} \\
 &=r^{l+1}\left(\frac{1-r^{-\floor{\frac{l+1}{\ceil{\frac{k}{2}}+1}}\left(\ceil{\frac{k}{2}}+1\right)}}{r^{\ceil{\frac{k}{2}}+1}-1}\right)\\
 &=\frac{r^{l+1}-r^{l+1-\floor{\frac{l+1}{\ceil{\frac{k}{2}
}+1}}\left(\ceil{\frac{k}{2}}+1\right)}}{r^{\ceil{\frac{k}{2}}+1}-1}.
\end{align*}
\end{proof}

Now we will show $\CV{\Trl}=|V'|$. For each $\vij \in V'$, $T_{\vij}^{\ceil{\frac{k}{2}}}$ is not in a failure state. We will also show that for any $\vij,\vipjp \in V'$ with $\vij \neq \vipjp$, $T_{\vij}^{\ceil{\frac{k}{2}}}$ and $ T_{\vipjp}^{\ceil{\frac{k}{2}}}$ are disjoint. Thus, for each vertex in $V'$, we need at least one vertex deletion to produce a subgraph of $\Trl$ which is in a failure state. We will also show that $\Trl-V'$ is in a failure state and, therefore, $V'$ forms a minimum $k$-diameter disconnecting set. For the sake of simplicity, we will denote $T_{\vij}^{\ceil{\frac{k}{2}}}$ with $T_{\vij}$.

\begin{theorem}
\label{VertexTrees}
Let $r, l$ and $k$ be positive integers. Let $\Trl=(V,E)$ be a perfect $r-ary$ tree with height $l$. Then, 
$$\CV{\Trl} = \frac{r^{l+1}-r^{l+1-\floor{\frac{l+1}{\ceil{\frac{k}{2}
}+1}}\left(\ceil{\frac{k}{2}}+1\right)}}{r^{\ceil{\frac{k}{2}}+1}-1}.$$
\end{theorem}

\begin{proof} 
Let $r, l$, and $k$ be positive integers. Let $\Trl=(V,E)$. 

First we will establish that $$\CV{\Trl} \geq \frac{r^{l+1}-r^{l+1-\floor{\frac{l+1}{\ceil{\frac{k}{2}
}+1}}\left(\ceil{\frac{k}{2}}+1\right)}}{r^{\ceil{\frac{k}{2}}+1}-1}$$ by finding a set of disjoint $k$-pairs. Let $$V'=\left\{v_{m \left(\ceil{\frac{k}{2}}+1\right),j} \in V: 1 \leq m \leq \floor{\frac{l+1}{\ceil{\frac{k}{2}}+1}}, 1\leq j \leq r^{l+1-m\left(\ceil{\frac{k}{2}}+1\right)}\right\}.$$ For each $v_{i,j} \in V'$ let $T_{v_{i,j}}$ be the subgraph induced on $v_{i,j}$ and all of its descendants at a distance at most $\ceil{\frac{k}{2}}$. Notice $i>\ceil{\frac{k}{2}}$ for each $v_{i,j} \in V'$ and $T_{v_{i,j}}$ is a perfect $r$-ary tree with height $\ceil{\frac{k}{2}}$ for all $v_{i,j} \in V'$. Since the diameter of a perfect $r$-ary tree of height $a$ is $2a$, we have $diam(T_{v_{i,j}})=2\ceil{\frac{k}{2}}$ for each $v_{i,j} \in V'$. This implies, since $k$ is an integer, $$k+1\geq diam(T_{v_{i,j}}) \geq k.$$ Since $diam(T_{v_{i,j}}) \geq k$, $T_{v_{i,j}}$ contains at least one $k$-pair and, thus, is not in a failure state.

Let $\vij, \vijp \in V'$ with $j \neq j'$. We will show that all $T_{\vij}$ and $T_{\vijp}$ are disjoint. Since $\vij$ and $\vijp$ have a common ancestor, they cannot share any decendants, or $\Trl$ would contain a cycle and is not a tree. Therefore, $T_{\vij}$ and $T_{\vijp}$ are disjoint. 

Consider $\vij, \vipjp \in V'$ where $i \neq i'$. We will show that $T_{\vij}$ and $T_{\vipjp}$ are disjoint. By the definition of $V'$, vertices on different levels in $V'$ are at a distance of at least $\ceil{\frac{k}{2}}+1$ from each other, so $v_{i,j} \not \in T_{\vipjp}$ and $\vipjp \not \in T_{\vij}$. Since $\vij$ and $\vipjp$ are the roots of $T_{\vij}$ and $T_{\vipjp}$ respectively, and $T_{\vij}$ and $T_{\vipjp}$ are trees of height $\ceil{\frac{k}{2}}$, this implies $T_{\vij}$ and $T_{\vijp}$ are disjoint.

Therefore, we have shown that $\{T_{\vij}: \vij \in V'\}$ are pairwise disjoint.

Since $T_{v_{i,j}}$ contains a $k$-pair for each $\vij \in V'$, there are at least $|V'|$ disjoint $k$-pairs in $\Trl$. Therefore, by Corollary \ref{lowerbound} and Lemma \ref{cardvprimevertex}, 

$$\CV{\Trl} \geq \frac{r^{l+1}-r^{l+1-\floor{\frac{l+1}{\ceil{\frac{k}{2}}+1}}(\ceil{\frac{k}{2}}+1)}}{r^{\ceil{\frac{k}{2}}+1}-1}.$$ 

Now we will show that $$\CV{\Trl} \leq \frac{r^{l+1}-r^{l+1-\floor{\frac{l+1}{\ceil{\frac{k}{2}
}+1}}\left(\ceil{\frac{k}{2}}+1\right)}}{r^{\ceil{\frac{k}{2}}+1}-1}$$ by showing $\Trl - V'$ is in a failure state.

Consider $\Trl-V'$. By deleting all vertices in $V'$ from $\Trl$, we are deleting entire levels of vertices. In fact, we are deleting all vertices which are on a level which is an integer multiple of $\ceil{\frac{k}{2}}+1$. Notice, then, that $\Trl-V'$ is a disconnected graph where each component is a perfect $r$-ary tree. All of these trees have height $\ceil{\frac{k}{2}}-1$, except the tree containing the root vertex, which could have less. Then, $\di{T_{r,\ceil{\frac{k}{2}}-1}}=2(\ceil{\frac{k}{2}}-1) \leq 2(\frac{k+1}{2})-2 \leq k-1$. Therefore, each component has diameter less than $k$, and $G-V'$ is in a failure state. Hence, $$\CV{\Trl} \leq \frac{r^{l+1}-r^{l+1-\floor{\frac{l+1}{\ceil{\frac{k}{2}}+1}}\left(\ceil{\frac{k}{2}}+1\right)}}{r^{\ceil{\frac{k}{2}}+1}-1}.$$

Combining these two inequalities, we see that 
$$\CV{\Trl}=\frac{r^{l+1}-r^{l+1-\floor{\frac{l+1}{\ceil{\frac{k}{2}}+1}}\left(\ceil{\frac{k}{2}}+1\right)}}{r^{\ceil{\frac{k}{2}}+1}-1}.$$
\end{proof}

%**************************************************************
%**************************************************************
%**************************************************************

%Mixed Deletion Starts Here

%**************************************************************
%**************************************************************
%**************************************************************

\section{Mixed Deletion Results}

Now we will investigate the $k$-diameter component connectivity function of a few simple graph classes. We provide results for path graphs, cycles, complete graphs, and complete bipartite graphs. As with the previous section, we will assume throughout that $k \geq 2$ is a positive integer.

\subsection{Path Graphs}

To decompose path graphs into failure states, we will create path graphs components of maximum length which are in a failure state.

\begin{theorem}
	\label{mixedpaths}
	Let $P_n$ be the path on $n$ vertices. For any nonnegative integer $p \leq \CV{P_n}$,
	
	$$\CM{P_n,p}=\begin{cases} 
					 0 & \text{if } p=\floor{\frac{n}{k+1}} \\
					 \\
					 \floor{\frac{n-p(k+1)-1}{k}} & \text{if } p < \floor{\frac{n}{k+1}}.\\
 \end{cases}$$

	\end{theorem}
	
	\begin{proof} 

Let $n$ be a positive integer. Consider $P_n$ and let $p\leq \CV{P_n}$ be a nonnegative integer. If $p=\floor{\frac{n}{k+1}}$, then by Theorem \ref{vertexpath}, $\CM{P_n, p}=0$.

	Assume $p < \floor{\frac{n}{k+1}}$. Label vertices of $P_n$ consecutively starting at a pendant vertex so that $V=\{v_i : i\in \ZZ, 1 \leq i \leq n\}$. Let $V' \subset V$ such that $V'=\{v_{j(k+1)} : j \in \ZZ, 1 \leq j \leq p\}$. Then $P_n-V'$ has $p$ components with diameter $k-1$ and one component, denoted $C$, which is a path on $n-p(k+1)$ vertices. By Theorem 3.1 in \cite{ShankBuzz}, $\CE{C}=\floor{\frac{n-p(k+1)-1}{k}}$.
	Let $E'$ be the set of $\floor{\frac{n-p(k+1)-1}{k}}$ edges removed from component $C$ to produce a failure state. Then $P_n-V'-E'$ is in a failure state. Therefore, $\CM{P_n,p} \leq \floor{\frac{n-p(k+1)-1}{k}}$.

	Assume we are going to delete a set, $P$, of vertices where $|P|=p$ and a set, $Q$, of edges where $|Q|= q$ so that $P_n - P - Q$ is in a failure state. Since $P_n$ is a tree, any edge or vertex deletion creates at most one new component. Therefore, $P_n - P - Q$ can have at most $p+q+1$ components. If $P_n-P-Q$ is in a failure state, then each component contains at most $k$ vertices. Thus, the number of vertices in our original graph is the sum of all the vertices in the failed components plus the $p$ vertices we had to delete from $P$. Therefore, the number of vertices in our original graph would be bounded above by $(p+q+1)k +p.$ We will now use this fact to show that $\CM{P_n,p} =\floor{\frac{n-p(k+1)-1}{k}}$. 
	
	Assume by way of contradiction that there exists some set $V^* \subseteq V$ with $|V^*|=p$ and $E^* \subseteq E$ with $|E^*|=\floor{\frac{n-p(k+1)-1}{k}}-1$ such that $P_n-V^*-E^*$ is in a failure state. As shown above, if $P_n-V^*-E^*$ is in a failure state, then the number of vertices in $P_n$ is bounded above by 
$$|V(P_n)| \leq \left(p+\left(\floor{\frac{n-p(k+1)-1}{k}}-1\right)+1\right)k+p.$$

Simplifying this expression, we see
\begin{align*}
|V(P_n)| & \leq p(k+1)+k\floor{\frac{n-p(k+1)-1}{k}} \\
&\leq p(k+1)+k\left(\frac{n-p(k+1)-1}{k}\right)\\
& =n-1.
\end{align*}

Hence, we have a contradiction, because $|V(P_n)|=n$. 	
	
	\end{proof}

\subsection{Cycle Graphs} A cycle is only one vertex deletion away from becoming a path graph. Therefore, we can use our results for path graphs to analyze cycle graphs. 
	\begin{theorem}
	\label{mixedcycles}
	Consider $C_n$, the cycle on $n$ vertices. For any positive integer $n$ and nonnegative integer $p\leq \CV{C_n}$, 

$$\CM{C_n,p}=\begin{cases} 
					 0 & \text{if } k>\floor{\frac{n}{2}} \\
					 0 & \text{if }  k \leq \floor{\frac{n}{2}} \text{ and } p=\floor{\frac{n+k}{k+1}} \\
					 \floor{\frac{n-p(k+1)-1}{k}}+1 & \text{if } k \leq \floor{\frac{n}{2}} \text{and } 0< p < \floor{\frac{n+k}{k+1}}.\\
 \end{cases}$$
	\end{theorem}
	
	\begin{proof} 
	Let $C_n$ be the cycle graph on $n$ vertices. Note that $\di{C_n}=\floor{\frac{n}{2}}$, so if $k >\floor{\frac{n}{2}}$, $C_n$ is already in a failure state and requires no deletions. Consider when $p=\floor{\frac{n+k}{k+1}}$. Then by Theorem \ref{vertexcycle}, there exists a set $V'$ of $p$ vertices such that $C_n-V'$ is in a failure state. Hence, if $p=\floor{\frac{n+k}{k+1}}$, then $\CM{C_n,p}=0$. 
	
	Consider when $k \leq \floor{\frac{n}{2}}$ and $0< p < \floor{\frac{n+k}{k+1}}$. Note that any vertex deletion leaves a path on $n-1$ vertices, hence, $\CM{C_n,p}=\CM{P_{n-1}, p-1}$. Therefore, by Theorem \ref{mixedpaths} $\CM{C_n,p}=\floor{\frac{(n-1)-(p-1)(k+1)-1}{k}}$. Simplifying this expression, we see 
	$$\CM{C_n,p}=\floor{\frac{n-p(k+1)-1}{k}}+1.$$
	\end{proof} 
	
\subsection{Complete Graphs}

Consider the complete graph on $n$ vertices, denoted $K_n$. Since any complete graph has diameter 1, we see $K_n$ is already in a failure state since we are assuming $k\geq 2$. Thus, we see the following obvious result:
	\begin{theorem}
	For every positive integer $n$, $$\CM{K_n, p}=0.$$
\end{theorem}

\subsection{Complete Bipartite Graphs}
\begin{theorem}
Consider the complete bipartite graph $K_{(a,b)}$ where $a$ and $b$ are positive integers with $a \leq b$. Then for any nonnegative integer $p \leq \CV{K_{a,b}}$,

$$\CM{K_{a,b},p}=\begin{cases} 
					 0 & \text{if } a=b=1 \\
					 0 & \text{if } k>2 \text{ and } b > 1\\
					 (a-p)(b-1) & \text{if } k=2 \text{ and } b >1. \\
 \end{cases}$$
\end{theorem}

\begin{proof} 

Let $k$ be a positive integer. Consider the complete bipartite graph, denoted $K_{a,b}=(V,E)$ with parts $A$ and $B$ where $V=A \cup B$, $A \cap B =\emptyset$, $|A|=a>0$, and $|B|=b>0$. Furthermore, assume without loss of generality that $a \leq b$. 

If $a=b=1$, then $\di{K_{a,b}}=1$ and $K_{a,b}$ is in a failure state.
If $b \geq 2$, then $\di{K_{a,b}}=2$. If $k>2$, then $K_{a,b}$ is in a failure state.
If $k=2$, then $K_{a,b}$ is not in the failure state. Therefore, fix $k=2$ for the remainder of the proof. 

 Note that min$\{a,b\}=a$, so Theorem \ref{vertexcompbipart} implies that $\CM{K_{a,b},a}=0$. Thus, assume $p<a$. Let $V'$ be a set of vertices deleted from $K_{a,b}$ such that $|V'|=p$.
We claim that $V' \subseteq A$ is an optimal $p$-set. In other words, it is optimal to delete all $p$ vertices from part $A$ of $K_{a,b}$. Then by, Theorem 3.3 of \cite{ShankBuzz}, $\CM{K_{(a,b)}}=\CE{K_{a-p,b}}=(a-p)(b-1)$.

Assume by way of contradiction that $V' \subseteq A$ is not an optimal $p$-set. Then there exists an optimal $p$-set $V^*=V_A \cup V_B$, where $V_A \subseteq A$, $V_B \subseteq B$, $|V_A|=x$, $|V_B|=y$, $y \geq 1$ and $x+y=p$. In other words, it is optimal to delete some vertices from part $A$ and some from part $B$ of $K_{a,b}$. Then, we have the following two cases.

\noindent Case 1: Assume $a-x \leq b-y$. Then if $V^*$ is an optimal $p$-set, by Theorem 3.3 of \cite{ShankBuzz}, $\CM{K_{a,b},p}=\CE{K_{a-x,b-y}}=(a-x)(b-y-1)$.

\noindent Note that $y \geq 1$ implies 
$$y(a-x) \leq y(b-1).$$
Then, by substituting $y=p-x$ on the right side of this inequality, we see
$$x(b-1)+y(a-x) \leq p(b-1).$$
Adding $a(b-1)$ to both sides of the above inequality and simplifying, we see
$$(a-p)(b-1) \leq (a-x)(b-y-1).$$

Hence, $\CE{K_{a-p,b}} \leq \CE{K_{a-x,b-y}}$.

\noindent Case 2: Assume $b-y <a-x$. Then if $V^*$ is an optimal $p$-set, by Theorem 3.3 of \cite{ShankBuzz}, $\CM{K_{a,b},p}=\CE{K_{a-x,b-y}}=(b-y)(a-x-1)$. 

\noindent Note that $a \leq b$ implies that $a-x \leq b$. Multiplying by $(y-1)$ and then adding $b(a-x)$ to both sides of this inequality, we see
$$(a-x)b+y(a-x)-(a-x) \leq (a-x)b+y(b-1)-b+y.$$ 
Simplifying this expression, we see
$$(a-x-y)(b-1) \leq (b-y)(a-x-1).$$
Then substituting $p=x+y$, we see
$$(a-p)(b-1) \leq (b-y)(a-x-1).$$ 

Hence, $\CE{K_{a-p,b}} \leq \CE{K_{a-x,b-y}}$.

In either of these two cases, $\CE{K_{a-p,b}}=(a-p)(b-1) \leq \CE{K_{a-x,b-y}}$. Hence, $V'$ is an optimal $p$-set and

$$\CM{K_{a,b},p}=\begin{cases} 
					 0 & \text{if } a=b=1 \\
					 0 & \text{if } k>2 \text{ and } b > 1\\
					 (a-p)(b-1) & \text{if } k=2 \text{ and } b >1. \\
 \end{cases}$$
 
 \end{proof}

\nocite{*}
%\bibliographystyle{amsplain}

%\bibliography{MyBibFile2}

\providecommand{\bysame}{\leavevmode\hbox to3em{\hrulefill}\thinspace}
\providecommand{\MR}{\relax\ifhmode\unskip\space\fi MR }
\providecommand{\MRhref}[2]{%
  \href{http://www.ams.org/mathscinet-getitem?mr=#1}{#2}
}
\providecommand{\href}[2]{#2}

\end{document}